\documentclass{amsart}
\usepackage{a4wide,mathtools}

\newcommand{\defeq}{\coloneqq}
\newcommand{\IN}{\mathbb N}
\newcommand{\IZ}{\mathbb Z}
\newcommand{\IR}{\mathbb R}
\newcommand{\IQ}{\mathbb Q}
\newcommand{\Ra}{\Rightarrow}
\newcommand{\w}{\omega}

\newtheorem{theorem}{Theorem}
\newtheorem{lemma}{Lemma}
\newtheorem{claim}{Claim}
\newtheorem{proposition}{Proposition}

\theoremstyle{definition}

\newtheorem{definition}[theorem]{Definition}

\title{Midconvex sets in Abelian groups}
\author{Iryna Banakh, Taras Banakh, Maria Kolinko, Alex Ravsky}
\address{I.~Banakh, A.~Ravsky: Institute for Applied Problems of Mechanics and Mathematics of National Academy of Sciences of Ukraine,  Naukova 3b, Lviv, Ukraine}
\email{ibanakh@yahoo.com, alexander.ravsky@uni-wuerzburg.de}
\address{T.~Banakh, M.~Kolinko: Faculty of Mechanics and Mathematics, Ivan Franko National University of Lviv, Ukraine}
\email{t.o.banakh@gmail.com, marybus20@gmail.com}

\begin{document}
\begin{abstract} A subset $X$ of an Abelian group $G$ is called {\em midconvex} if for every $x,y\in X$ the set $\frac{x+y}2=\{z\in G:2z=x+y\}$ is a subset of $X$. We prove that a subset $X$ of an Abelian group $G$ is midconvex if and only if for every $g\in G$ and $x\in X$, the set $\{n\in\IZ:x+ng\in X\}$ is equal to $C\cap H$ for some order-convex set $C\subseteq \IZ$ and some subgroup $H\subseteq \IZ$ such that the quotient group $\IZ/H$ has no elements of even order. This characterization implies that  a subset $X$ of a periodic Abelian group $G$ is midconvex if and only if for every $x\in X$ the set $X-x$ is a subgroup of $G$ such that every element of the quotient group $G/(X-x)$ has odd order. Also we prove that a nonempty set $X$ in a subgroup $G\subseteq\IQ$ is midconvex if and only if $X=C\cap(H+x)$ for some order-convex set $C\subseteq\IQ$, some $x\in X$ and some subgroup $H$ of $G$ such that the quotient group $G/H$ contains no elements of even order. 
\end{abstract}
\maketitle

In this paper we study midconvex subsets of Abelian groups. Such subsets often appear in the theory of Functional Equations \cite[\S11]{BJ}, \cite[\S10]{BO10}, \cite{BO17}, \cite[p.111]{Kuczma} and Combinatorics of Groups \cite{BBKR}.

{\em All groups in this paper are assumed to be Abelian}.

\begin{definition} A subset $X$ of a group $G$ is called {\em midconvex} if for every $x,y\in X$, the set $$\frac{x+y}2\defeq\{z\in G:2z=x+y\}$$ is a subset of $X$.
\end{definition}

The aim of this paper is to characterize midconvex sets in groups that belong to certain special classes, in particular, in periodic groups and also in subgroups of the additive group of rational numbers. 

 Let $G$ be a subgroup of the real line $\IR$.
 A subset $X\subseteq G$ is called {\em order-convex in} $G$ if for every $x,y\in X$, the order interval set $\{g\in G:x\le g\le y\}$ is a subset of $X$.

\begin{lemma}\label{l:1} If a subset $X$ of a group $G$ is midconvex, then for every $x,y\in X$, the set $$Z\defeq\{n\in\IZ:x+n(y-x)\in X\}$$ is order-convex in $\IZ$.
\end{lemma}

\begin{proof} It follows from $\{x,y\}\subseteq X$ that $\{0,1\}\subseteq Z$. Assuming that $Z$ is not order-convex, we can find integer numbers $k<n<m$ such that $k,m\in Z$ and $n\notin Z$. Two cases are possible.
\smallskip

1. If $n>0$, then we can assume that $n$ is the smallest positive number such that $n\notin Z$.  The minimality of $n$ ensures that $\{0,1\}\subseteq \{0,\dots,n-1\}\subseteq Z$. By induction, we shall prove that $\{n,\dots,i\}\cap Z=\emptyset$ for every integer $i\ge n$. For $i=n$, this follows from the choice of $n$. Assume that for some $i>n$ we know that $\{n,\dots,i-1\}\cap Z=\emptyset$. Let $a\in\{n-2,n-1\}\subseteq Z$ be a unique  number such that $a+i$ is even and hence $a+i=2c$ for some $c\in\IZ$. Assuming that $i\in Z$, we obtain that $\{x+a(y-x),x+i(y-x)\}\subseteq X$. Since $2(x+c(y-x))=2x+(a+i)(y-x)=(x+a(y-x))+(x+i(y-x))$, the point $x+c(y-x)$ is an element of the midconvex set $X$ and hence $c\in Z$. On the other hand, $c=\frac12(a+i)<i$ and $c=\frac12(a+i)\ge \frac12(n-2+n+1)=n-\frac12$ and hence $c\in\{n,\dots,i-1\}\subseteq \IZ\setminus Z$, which is a desired contradiction showing that $i\in Z$. By the Principle of Mathematical Induction, $i\notin Z$ for all $i>n$. In particular, $m\notin Z$, which contradicts the choice of $m$.
\smallskip

2. If $n<0$,   then we can assume that $n$ is the largest negative number such that $n\notin Z$.  The maximality of $n$ and the inclusion $\{0,1\}\subseteq Z$ ensure that $\{n+1,n+2\}\subseteq Z$. By downward induction, we shall prove that $\{i,\dots,n\}\cap Z=\emptyset$  for every integer $i\le n$. For $i=n$ this follows from the choice of $n$. Assume that for some $i<n$ we know that $\{i+1,\dots,n\}\cap Z=\emptyset$. Let $a\in\{n+1,n+2\}\subseteq Z$ be a unique  number such that $a+i$ is even and hence $a+i=2c$ for some $c\in\IZ$. Assuming that $i\in Z$, we obtain that $\{x+i(y-x),x+a(y-x)\}\subseteq X$. Since $2(x+c(y-x))=2x+(i+a)(y-x)=(x+i(y-x))+(x+a(y-x))$, the point $x+c(y-x)$ is an element of the midconvex set $X$ and hence $c\in Z$. On the other hand, $c=\frac{i+a}2>i$ and $c=\frac{i+a}2\le \frac12(n-1+n+2)=n+\frac12$ and hence $c\in\{i+1,\dots,n\}\subseteq \IZ\setminus Z$, which is a desired contradiction showing that $i\in Z$. By the Principle of Mathematical Induction, $i\notin Z$ for all $i<n$. In particular, $k\notin Z$, which contradicts the choice of $k$.
\end{proof}

Now we can prove the promised  characterization of midconvex sets in groups. We recall that the {\em order} of an element $g$ of a group $G$ is the smallest number $n\in\IN$ such that $ng=0$. If $ng\ne 0$ for all $n\in\IN$, then we say that $g$ has infinite order.

\begin{theorem}\label{t:1} A subset $X$ of a group $G$ is midconvex if and only if for every $g\in G$ and $x\in X$, the set $\{n\in\IZ:x+ng\in X\}$ is equal to $C\cap H$ for some order-convex set $C\subseteq \IZ$ and some subgroup $H\subseteq \IZ$ such that the quotient group $\IZ/H$ has no elements of even order.
\end{theorem}

\begin{proof} First assume that the set $X$ is not midconvex in $G$. Then there exist points $x,y\in X$ and $g\in G\setminus X$ such that $2z=x+y$. Consider the set $Z\defeq\{n\in\IZ:x+ng\in X\}$ and observe that $\{0,2\}\subseteq Z$ and $1\notin Z$. Assume that $Z=C\cap H$ for some order-convex set $C$ in $\IZ$ and some subgroup $H$ of $\IZ$ such that the quotient group $\IZ/H$ has no elements of even order. It follows from $\{0,2\}\subseteq Z\subseteq H$ that $2\IZ\subseteq H$. Since $\IZ/H$ has no elements of even order, $H\ne 2\IZ$ and hence $H=\IZ$. Since $\{0,2\}\subseteq C\cap H=C$, the order-convexity of $C$ implies that $1\in C=C\cap H=Z$, which is a desired contradiction completing the proof of the ``if'' part.
\smallskip

To prove the ``only if'' part, assume that the set $X\subseteq G$ is midconvex and take any $x\in X$ and $g\in G$. If $Z\defeq\{n\in\IZ:x+ng\in X\}=\{0\}$, then $Z=C\cap H$ for the order-convex set $C\defeq\{0\}$ and the subgroup $H\defeq\IZ$. The quotient group $\IZ/H$ is trivial and hence has no elements of even order. So, assume that $Z\ne \{0\}$. In this case there exists a nonzero integer number $m\in Z$. We can assume that its absolute value $|m|$ is the smallest possible, that is $i\notin Z$ for all $i\in\IZ$ with $0<|i|<|m|$. Assuming that $m$ is even, consider the element $z=x+\frac{m}2g\in G$ and observe that $m\in Z$ implies $x+mg\in X$. Since $2z=2x+mg=x+(x+mg)$, the midconvexity of $X$ guarantees that $z\in X$ and hence $\frac m2\in Z$, which contradicts the choice of $m$. This contradiction shows that $m$ is odd. Then for the subgroup $H\defeq m\IZ$, the quotient group $\IZ/H$ contains no elements of even order. By Lemma~\ref{l:1}, the set $\{n\in\IZ:x+nmg\in X\}$ is order-convex in $\IZ$ and hence the set $S\defeq \{n\in H:x+ng\in X\}$ is order-convex in $H$. Then there exists an order-convex set $C$ in $\IZ$ such that $S=C\cap H$. It remains to show that $Z=C\cap H$. It is clear that $C\cap H=\{n\in H:x+ng\in X\}\subseteq\{n\in\IZ:x+ng\in X\}=Z$. Assuming that $Z\ne C\cap H$, we can find a number $k\in Z\setminus C\cap H$. We can assume that $|k|$ is the smallest possible, and hence every $i\in Z$ with $|i|<|k|$ belongs to $C\cap H$. If $k$ is even, then $\{x,x+kg\}\subseteq  X$ and the midconvexity of $X$ implies $x+\frac k2g\in Z$ and hence $\frac k2\in H$ by the minimality of $k$. Then $k\in H$  and $k\in S=C\cap H$, which contradicts the choice of $k$. This contradiction shows that $k$ is odd. The choice of  $|m|$ ensures  that $|m|<|k|$. Since the numbers $m$ and $k$ are odd, the number $\frac{m+k}2$ is integer and $|\frac{m+k}2|<|k|$. The midconvexity of $Z\ni m,k$ implies that $\frac{m+k}2\in Z$ and hence $\frac{m+k}2\in C\cap H$ by the choice of $k$. Then $k=2\frac{m+k}2-m\in H$ and hence $k\in S=C\cap H$, which contradicts the choice of $k$. In both cases we obtain a contradiction showing that $Z=C\cap H$.
\end{proof}

A group $G$ is called {\em periodic} if for every $g\in G$ there exists $n\in\IN$ such that $ng=0$. 

\begin{theorem}\label{t:2} A subset $X$ of a periodic group $G$ is midconvex if and only if for every $x\in X$ the set $X-x$ is a subgroup of $G$ such that every element of the quotient group $G/(X-x)$ has odd order.
\end{theorem}

\begin{proof} To prove the ``only if'' part, assume that the subset $X$ is midconvex in $G$. Fix any element $x\in X$.

\begin{claim}\label{cl:2a} For every $a\in X-x$ we have $2a\in X-x$.
\end{claim}

\begin{proof} By Theorem~\ref{t:1}, for the set $Z\defeq\{n\in\IZ:x+na\in X\}$ there exists an order-convex set $C$ in $\IZ$ and a subgroup $H$ of $\IZ$ such that $Z=C\cap H$ and the quotient group $\IZ/H$ has no elements of even order. It follows from $x\in X$ and $a\in X- x$ that $\{0,1\}\subseteq Z=C\cap H\subseteq H$ and hence $H=\IZ$ and $Z=C\cap H=C$. Since the group $G$ is periodic, there exists $m\in\IN$ such that $ma=0$ and hence  $Z=Z+m$ and $C=C+m$. The order-convexity of $C=C+m$ implies $C=\IZ$ and hence $Z=C=\IZ$ and $2\in Z$, which implies $x+2a\in X$ and $2a\in X-x$.
\end{proof}

\begin{claim} The set $X-a$ is a subgroup of $G$.
\end{claim}

\begin{proof} Given any elements $a,b\in X-a$, we have $\{2a,2b\}\subseteq X-x$, according to Claim~\ref{cl:2a} and then $a+b+x=\frac12((2a+x)+(2b+x))\in X$ by the midconvexity of $X$. Since $a+b\in X-x$, the set $X-x$ is a subsemigroup of $G$. Since $G$ is periodic, the semigroup $X-x\subseteq G$ is a subgroup  of $G$.
\end{proof}

Assuming that some element of the quotient group $G/(X-x)$ has even order, we can find an element $y\in G\setminus(X-x)$ such that $2y\in X-x$. By the midconvexity of $X$, $y+x=\frac{x+(2y+x)}2\in X$ and hence $y\in X-x$, which contradicts the choice of $y$.
This completes the proof of the ``only if'' part of Theorem~\ref{t:2}.
\smallskip

To prove the ``if'' part, assume that for every $x\in X$ the set $X-x$ is a subgroup of $G$ such that the quotient group $G/(X-x)$ contains no elements of even order. To derive a contradiction, assume that the set $X$ is not midconvex in $G$. Then there exist elements $x,y\in X$ and $z\in G\setminus X$ such that $2z=x+y$. By our assumption, $X-x$ is a subgroup such that the quotient group $G/(X-x)$ has no elements of even order. Since $2(z-x)=y-x\in X-x$, the element $(z-x)+(X-x)$ of the quotient group $G/(X-x)$ has order at most 2. Since this order is not even, $z-x$ has order $1$ and hence belongs to the subgroup $X-x$. Therefore, $z-x\in X-x$, which contradicts the choice of $z$.
\end{proof}

Our final theorem describes midconvex sets in subgroups of the ordered group $\IQ$.

\begin{theorem}\label{t:3} Let $G$ be a subgroup of $\IQ$. A nonempty set $ X\subseteq G$ is midconvex if and only if $X=C\cap(H+x)$ for some order-convex set $C\subseteq\IQ$, some $x\in X$ and some subgroup $H$ of $G$ such that the quotient group $G/H$ contains no elements of even order. 
\end{theorem}

\begin{proof} To prove the ``if'' part, assume that $X=C\cap(H+x)$ for some order-convex set $C\subseteq\IQ$, some $x\in X$ and some subgroup $H$ of $G$ such that the quotient group $G/H$ contains no elements of even order. To see that $X$ is mid-convex, take any elements $a,b\in X$ and $c\in\frac{a+b}2\subseteq G$. We lose no generality assuming that $a\le b$ and hence $a\le c\le b$. Since $a,b\in X\subseteq C$, the order convexity of $C$ ensures that $c\in C$. Observe that $2(c-x)=(a-x)+(b-x)\in (X-x)+(X-x)\subseteq H+H=H$ and hence the element $(c-x)+H$ of the quotient group $G/H$ has order at most 2. Since the group $G/H$ contains no elements of even order, $c-x\in H$ and hence $c\in C\cap(H+x)=X$, witnessing that $X$ is midconvex.
\smallskip

To prove the ``only if'' part, assume that the nonempty set $X$ is midconvex in $G$. Fix any element $x\in X$.

If $X=\{x\}$ is a singleton, then $X=C\cap (H+x)$ for the order-convex set $C=\{x\}$ and the subgroup $H=G$ whose quotient group $G/H$ is trivial and hence contains no elements of even order. So, we assume that $X$ contain at least two distinct elements. Fix any elements $x,x'\in X$ with $x<x'$.

 Write the countable group $G\subseteq\IQ$ as the union $G=\bigcup_{n\in\w}G_n$ of an increasing sequence $(G_n)_{n\in\w}$ of finitely generated subgroups such that $\{x,x'\}\subseteq G_0$. The finitely generated subgroup $G_n$ of $\IQ$ is contained in a cyclic subgroup of $\IQ$ and hence $G_n$ is cyclic. Then $G_n=g_n\IZ$ for a unique positive rational number $g_n$. 
 
 For every $n\in\w$, the midconvexity of the set $X$ in $G$ implies the midconvexity of the set $X\cap G_n$ in the cyclic group $G_n=g_n\IZ$. By Theorem~\ref{t:1}, there exists an order-convex set $C_n\subseteq g_n\IZ$ and a subgroup $H_n\subseteq G_n$ such that $X\cap G_n=C_n\cap (H_n+x)$ and the quotient group $G_n/H_n$ contains no elements of even order. Since $\{x,x'\}\subseteq X\cap G_n\subseteq H_n+x$, the subgroup $H_n$ of $G_n=g_n\IZ$ is not trivial and hence $H_n=m_ng_n\IZ$ for some odd natural number $m_n$. Since the intersection of an arbitrary family of order-convex sets in the cyclic group $G_n=g_n\IZ$ is order-convex, we can assume that $C_n$ is the smallest order-convex subset of $G_n$ such that $X\cap G_n=C_n\cap(H_n+x)$.  
 
\begin{claim}\label{cl:CH} For every $n\in\w$, $C_n\subseteq  C_{n+1}$ and $H_n\subseteq  H_{n+1}$.
\end{claim}

\begin{proof} Observe that $C_n\cap (H_n+x)=X\cap G_n=(X\cap G_{n+1})\cap G_n=(C_{n+1}\cap (H_{n+1}+x))\cap G_n$. The minimality of the order-convex set $C_n$ ensures that for every $c\in C_n$ there exists $y,z\in C_n\cap X$ such that $y\le c\le z$. Then $y,z\in C_n\cap X\subseteq G_n\cap X\subseteq G_{n+1}\cap X\subseteq C_{n+1}$ and hence $c\in C_{n+1}$ by the order-convexity of $C_{n+1}$. This shows that $C_n\subseteq C_{n+1}$. 

Since $x'\in X\cap G_n= C_n\cap (x+H_n)=C_n\cap (x+m_ng_n\IZ)$ and $x<x'$, there exists $k'_n\in\IN$ such that $x'=x+m_ng_nk_n'\in C_n$. The order convexity of the set $C_n\ni x$ in $G_n$ implies that $x+m_ng_n\in C_n\cap (x+H_n)=X\cap G_n\subseteq X\cap G_{n+1}\subseteq H_{n+1}+x$ and hence $H_n=m_ng_n\IZ\subseteq H_{n+1}$.
\end{proof}

\begin{claim} The set $C=\bigcup_{n\in\w}C_n$ is order-convex in the group $G$.
\end{claim}

\begin{proof} Given any elements $x\le y\le z$ of the group $G=\bigcup_{n\in\w}G_n$ with $x,z\in C=\bigcup_{n\in\w}C_n$, find $m\in\w$ such that $x,y,z\in G_m$ and $x,y\in C_m$. Such  number $m$ exists because $(C_n)_{n\in\w}$ and $(G_n)_{n\in\w}$ are increasing sequences of sets. The order-convexity of the set $C_m$ in the group $G_m$ implies $y\in C_m\subseteq C$, witnessing that the set $C$ is order-convex in $G$.
\end{proof}

By Claim~\ref{cl:CH}, the sequence of subgroups $(H_n)_{n\in\w}$ is increasing and hence $H=\bigcup_{n\in\w}H_n$ is a subgroup of the group $G$.

\begin{claim} The quotient group $G/H$ has no elements of even order.
\end{claim}

\begin{proof} If the group $G/H$ has an element of even order, then it has an element of order 2. Consequently, there exists an element $g\in G\setminus H$ such that $2g\in H$. Find $n\in\w$ such that $g\in G_n$ and $2g\in H_n$. It follows from $H_n\subseteq H$ and $g\in G\setminus H$ that $g\in G_n\setminus H_n$ and hence $g+H_n$ is an element of order $2$ in the quotient group $G_n/H_n$, which contradicts the choice of the groups $G_n$ and $H_n$.
\end{proof}

Taking into account that $(C_n)_{n\in\w}$ and $(H_n)_{n\in\w}$ are increasing sequences of sets, we obtain that $$X=\bigcup_{n\in\w}X\cap G_n=\bigcup_{n\in\w}C_n\cap(H_n+x)=C\cap (H+x),$$
which completes the proof of Theorem~\ref{t:3}.
\end{proof}

\end{document}